\documentclass[final,leqno,onefignum]{siamltex704}

\usepackage{fancyhdr}
\usepackage[applemac]{inputenc}
\usepackage[T1]{fontenc}
\usepackage[english]{babel}
\usepackage{amsmath,amssymb,amsfonts,amsopn,mathrsfs,bm}
\usepackage{calc} 
\usepackage{cite}
\usepackage{color}
\usepackage{epsfig,verbatim}
\usepackage{subfigure}
\usepackage{mathtools}
\usepackage{latexsym}
\usepackage{pdfpages}
\usepackage{graphicx}
\usepackage{showkeys}
\usepackage{enumerate}
\usepackage{tikz,xifthen}
\usetikzlibrary{matrix}
\usepackage{hyperref}
\hypersetup{pdfstartview={XYZ null null 1.00}} 
\usepackage[parfill]{parskip}
\usepackage{epstopdf}
\usepackage[font=small,labelfont=bf]{caption}

\newtheorem{assumption}{Assumption}
\newtheorem{remark}{Remark}
\newtheorem{teo}{Theorem}[section]
\newtheorem{prop}[teo]{Proposition}

\newtheorem{coro}[teo]{Corollary}

\newcommand{\RR}{{\mathbb{R}}}
\newcommand{\der}{\partial}
\newcommand{\ep}{\varepsilon}
\newcommand{\epsi}{{\sigma}} 

\newcommand{\om}{\Omega}
\newcommand{\omze}{{\omega_0}}
\newcommand{\omun}{{\omega_1}}
\newcommand{\omti}{{\tilde{\om}}}

\begin{document}

\title{Stable determination of polyhedral interfaces from boundary
  data for the Helmholtz equation} 
\author{
Elena Beretta~\thanks{Dipartimento di Matematica "Brioschi",
   Politecnico di Milano, Milano, Italy
  (\texttt{elena.beretta@polimi.it})}
\and
Maarten V. de Hoop~\thanks{Department of Mathematics, Purdue
  University, West Lafayette, USA (\texttt{mdehoop@purdue.edu})}
\and
Elisa Francini~\thanks{Dipartimento di Matematica e Informatica
  ``U. Dini'', Universit\`{a} di Firenze, Italy
  (\texttt{elisa.francini@unifi.it})}
\and
Sergio Vessella~\thanks{Dipartimento di Matematica e Informatica
  ``U. Dini'', Universit\`{a} di Firenze, Italy
  (\texttt{sergio.vessella@unifi.it})}
}

\date{\today}

\maketitle


\begin{abstract}
We study an inverse boundary value problem for the Helmholtz equation
using the Dirichlet-to-Neumann map. We consider piecewise constant
wave speeds on an unknown tetrahedral partition and prove a
Lipschitz stability estimate in terms of the Hausdorff distance
between partitions.
\end{abstract}

\noindent
\textbf{Keywords.} Inverse boundary value problem, Helmholtz equation,
Lipschitz stability

\noindent
\textbf{MSC: }35R30, 35J08, 35J25

\section{Introduction}
\label{sec:1}

We consider an inverse boundary value problem for the Helmholtz
equation
\[
   \Delta u + \omega^2 q(x)u
       = 0\quad \textrm{in } \Omega\subset\RR^3 ,
\] 
where $q=c^{-2}$ and $c$ is the wavespeed. The data are the
Dirichlet-to-Neumann map and the objective is to recover the
wavespeed. The uniqueness of this inverse problem was established by
Sylvester and Uhlmann \cite{SU} for $q\in L^{\infty}(\Omega)$.
Concerning stability, conditional logarithmic continuous dependence of
the wavespeed on the Dirichlet-to-Neumann map has been proven in
\cite{A} in the case of wavespeeds in $H^s(\Omega)$ with
$s>\frac{3}{2}$. We refer to Novikov \cite{Nov} for a refinement of
this stability estimate. The logarithmic rate of stability is optimal
\cite{M}. For the inverse conductivity problem the authors of
\cite{AV} proposed restricting the class of unknown coefficients to a
finite dimensional set to obtain Lipschitz stability estimates. The
result was extended to complex-valued conductivities in \cite{BF}. In
this finite dimensional setting, in \cite {BdHQ, BdHQS}, a Lipschitz
stability estimate for the recovery of piecewise constant wavespeeds
for a given domain partition from boundary data for the Helmholtz
equation, and an estimate for the stability constant in terms of the
number of domains in the partition, were obtained.

Here, we study the problem of determining the finite partition from
boundary data given a (possibly large) finite set of attainable values
for the wavespeed. Due to the severe nonlinearity of the problem the
derivation of Lipschitz stability estimates is more subtle. For this
reason, we consider a partitoning of the domain with a (regular)
unstructured tetrahedral mesh. In fact, an unstructured tetrahedral
mesh admits a local refinement and, with piecewise constant
wavespeeds, can accurately approximate realistic models in
applications. In geophysics, we mention as an example the work of
R\"{u}ger and Hale \cite{RugerHaleGeophysics:2006}. Here, knowledge of
a set of attainable values for the wavespeed can be motivated by the
general knowledge of relevant rock types. The deformation allows one
to adjust the mesh and recover structures in the models. In
geodynamics, these structures can be an imprint of the local geology
and tectonics \cite{Hilst}. Moreover, one can parametrize major
discontinuities at (polyhedral) surfaces by connecting boundaries of
subdomains in the partition via a segmentation for example.

In this paper, we establish a Lipschitz stability estimate expressed
in terms of the Hausdorff distance between partitions using
tetrahedra from the Dirichlet-to-Neumann map. Lipschitz stability
estimates provide a framework for optimization, specifically,
iterative reconstruction of the wavespeed with a convergence radius
determined by the stability constant \cite{dHQS1, dHQS2}. The recovery
of polyhedral interfaces then becomes a shape optimization. The
analysis in \cite{dHQS1} makes explicit use of a Landweber
iteration. Via successive approximations, and making use of estimates
for the corresponding growth of the stability constant, the
reconstruction can be cast into a multi-level scheme \cite{dHQS2}
effectively enlarging the radius of convergence.
%
%
%
As an important application, we mention so-called time-harmonic full
waveform inversion (FWI) developed in reflection seismology
\cite{Pratt1999, Pratt1998, Pratt2004, Virieux2009} with the goal to
image wavespeed variations in Earth's interior. The data, here, are
essentially the single-layer potential operator. However, stability
estimates for the Dirichlet-to-Neumann map directly carry over to
stability estimates for this operator.

We give an outline of the paper.  We first state the main result and
the main assumptions (Section~\ref{sec:2}). Then we establish a rough
stability estimate for the potentials using complex geometrical optics
(CGO) solutions following the outline of an estimate in Beretta
\textit{et al.} \cite{BdHQS} (Section~\ref{sec:3}). The CGO solutions
were introduced by Sylvester and Uhlmann \cite{SU} in their proof of
uniqueness of this inverse boundary value problem. The CGO solutions
in our analysis differ slightly from theirs to obtain better constants
in the stability estimates as proposed in \cite{S}. We proceed with
establishing the recovery of the number of tetrahedra in the mesh from
the potential, and with expressing the Hausdorff distance between
meshes in terms of the difference of piecewise constant potentials
defined on these meshes. Naturally, the information on the Hausdorff
distance between meshes can be transformed to information on the
vertices of the tetrahedra forming the meshes
(Section~\ref{sec:4}). The main part of the proof of our result
pertains to obtaining a lower bound for the Gateaux derivative of the
Dirichlet-to-Neumann map under mesh deformation (Section~\ref{sec:5}).

\subsection*{Notation}

We use the Fourier transform convention,
\[
   \hat{f}(\xi)
        = \int_{\RR^3} f(x) e^{i x \cdot \xi} dx.
\]
If the function $f$ is defined on a subset of $\RR^3$, it is extended
to $\RR^3$ attaining the value zero. We denote by $\check{f}$ the
inverse Fourier transform of $f$,
\begin{equation}\label{invFT}
   \check{f}(x)
        = \frac{1}{(2\pi)^3} \int_{\RR^3} f(\xi)
              e^{-i x \cdot \xi} d\xi.
\end{equation}
We introduce coordinates, $x = (x^\prime,x_3)$, in $\RR^3$, where
$x^\prime\in\RR^{2}$ and $x_3\in \RR$. We denote the open ball in
$\RR^3$ centered at $x$ of radius $r$ by $B_r(x)$, and the open ball
in $\RR^2$ centered at $x^\prime$ of radius $r$ by
$B_r^\prime(x^\prime)$.

\section{Assumptions and main result}
\label{sec:2}

We let $\om$ be a bounded domain in $\RR^3$ such that $\RR^3 \setminus
\om$ is connected,
\begin{equation}\label{1.1}
   \om \subset B_R(0) \text{ for some }R > 0,
\end{equation}
and
\begin{equation}\label{1.1.5}
   \om \text{ has a Lipschitz boundary
                 with constants }r_0\text{ and }K_0,
\end{equation}
that is, for any point $P \in \der\om$, there exists a rigid
transformation of coordinates under which $P=0$ and
\[
   \om \cap \{(x^\prime,x_3) \in \RR^3 \,:\,
   |x^\prime|< r_0,\, |x_3| < K_0r_0\}
   = \{(x^\prime,x_3)\,:\, |x^\prime| < r_0,\,\,
                        x_3 > \psi(x^\prime)\},
\]
where $\psi$ is a Lipschitz continuous (level set) function in
$B_{r_0}^\prime$ such that
\[
   \psi(0) = 0\text{ and }\,
     \|\nabla\psi\|_{L^{\infty}(B_{r_0}^\prime)} \leq K_0.
\]

We consider the boundary value problem for the Helmholtz equation,
\begin{equation}\label{4.1}
\left\{\begin{array}{rcl}
             \Delta u+\omega^2 q u & = & 0\text{ in }\om, \\
             u & = & \phi \text{ on }\der\om
       \end{array}
\right.
\end{equation}
for $\phi \in H^{1/2}(\der\om)$, and introduce the
Dirichlet-to-Neumann map
\begin{equation}\label{3.6}
   \Lambda_{q}: H^{1/2}\left(\der\om\right)
                       \to H^{-1/2}\left(\der\om\right)
\end{equation}
according to
\begin{equation}\label{4.2}
   \phi \to \Lambda_{q}(\phi) := \left.
   \frac{\der u}{\der \nu}\right|_{\der\om}.
\end{equation}
The normal derivative is defined in the weak sense as
\[
   \left\langle {\frac{\der u}{\der \nu}},
       \psi|_{\der\om} \right\rangle
   = \int_\om (\nabla u \cdot \nabla \psi
                          -\omega^2 q u \psi) d x
\]
for every $\psi \in H^1(\om)$. In the above, $q \in
L^{\infty}(\der\om)$ is identified with $c^{-2}$ where $c$ denotes the
wavespeed. The solution of (\ref{4.1}) exists in
$H^{1}\left(\om\right)$ and is unique if $\omega$ is not in the
Dirichlet spectrum of $q^{-1} \Delta$ on $\om$.

We introduce $\omze$, $\omun$ such that $0 < \omze < \omun$ and
\begin{equation}\label{3.5}
   \omun \leq \sqrt{\frac{\lambda_1(B_R)}{2Q_0}},
\end{equation}
where $\lambda_1(B_R)$ is the first eigenvalue of $-\Delta$ on
$B_R$. We recall that $\lambda_1(B_R) = \lambda_1(B_1)R^{-2}$. (If we
detect the spectrum, we substitute the true first eigenfrequency for
$\omega_1$.) We then assume that
\begin{equation}\label{omega}
   \omze \leq \omega \leq \omun.
\end{equation}

\subsection*{Unstructured tetrahedral mesh}

We let $\{T_j\}_{j=1}^N$ be a regular partition of $\om$ into
tetrahedra, namely a collection of closed tetrahedra such that
\begin{equation}\label{1.4}
   \overline{{\om}} = \cup_{j=1}^NT_j ;
\end{equation}
\begin{eqnarray}\label{1.5}
&&\text{for }j\neq k\text{ either }T_j\cap T_k=\emptyset \text{ or it
    consists of a common vertex,}\\ &&\text{ a common edge or a
    common facet;}
\nonumber\end{eqnarray}
\begin{equation}\label{1.5.5}
\text{ the radius of the insphere of each tetrahedron is larger than
}r_1>0.
\end{equation}
We say that two different tetrahedra of such regular partition are
adjacent if they share a common facet.

\begin{figure}
\centering
\includegraphics[scale=0.50]{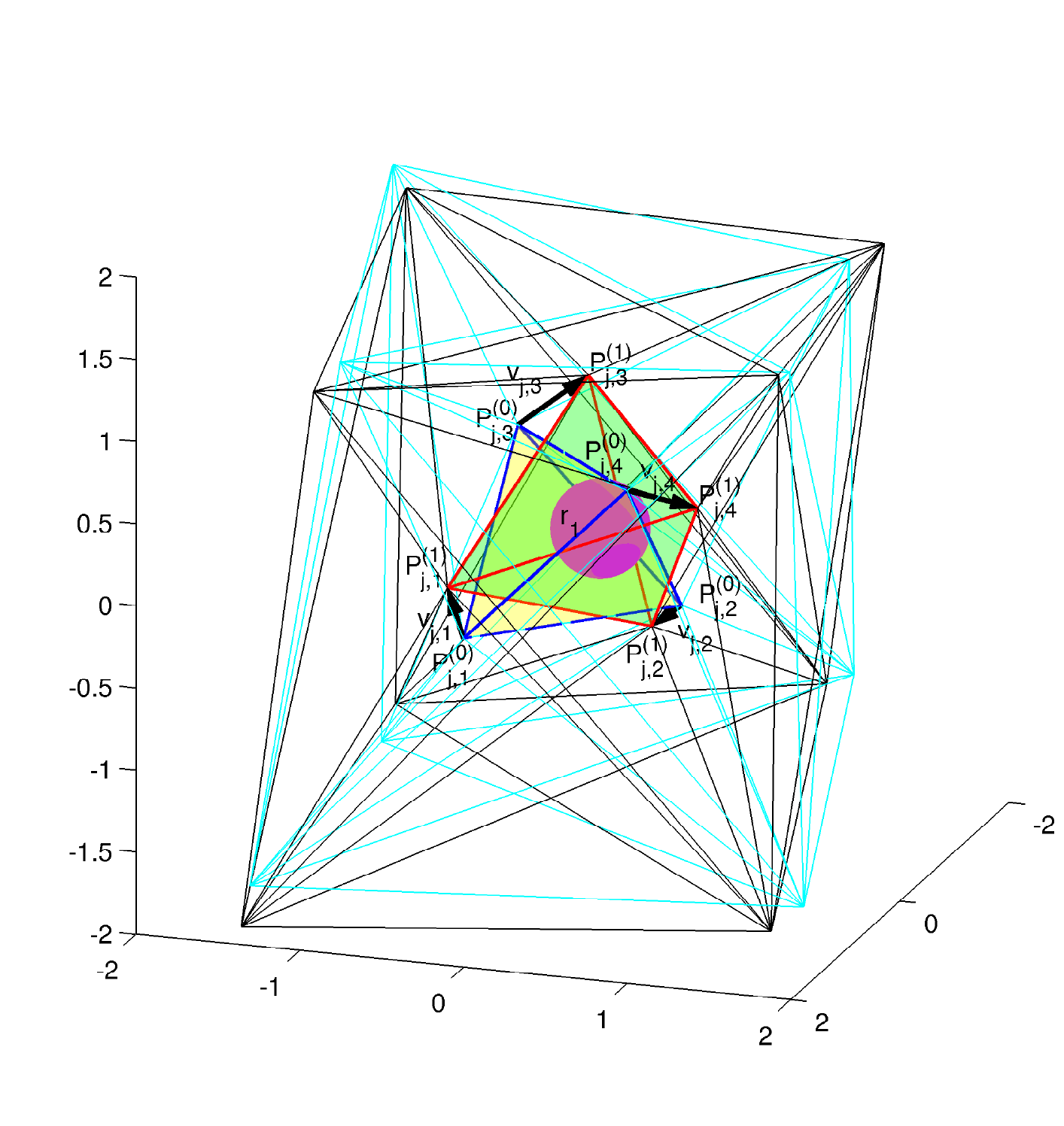}
\includegraphics[scale=0.22]{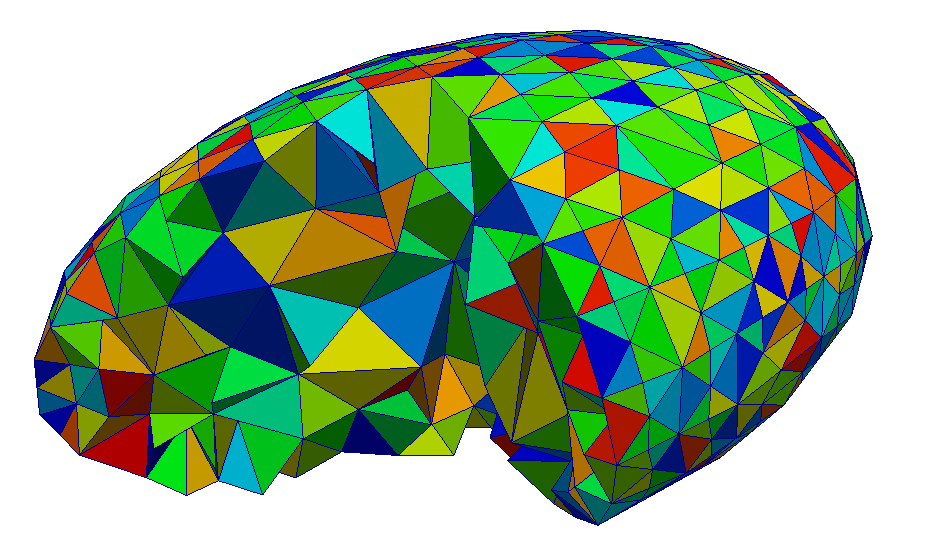}
\caption{Left: Quantities associated with the assumptions, and
  deformation of the mesh (cf.~(\ref{vertici})) Right: An example
  model, containing polyhedral interfaces, in the `stable' class.}
\label{fig:1}
\end{figure}

\begin{remark}\label{rem1}
Assumption \eqref{1.5.5}, together with \eqref{1.1} implies that the
tetrahedra of the partition are not degenerate. In particular, there
are two positive numbers $d_1$ and $\alpha_1$ (depending on $R$ and
$r_1$ only) such that
\begin{eqnarray}\label{2.1}
&&\text{for each }T_j \text{ the distance between vertices is greater
    than }d_1 \\
&&\text{ and internal angles of triangular facets are greater than
  }\alpha_1.
\nonumber
\end{eqnarray}
\end{remark}

Indeed, we point out that assumptions \eqref{1.5.5} and \eqref{1.1}
are equivalent to the following

\begin{assumption}\label{Ass:1}
There exists a positive constant $C_1$ such that
\begin{equation}\label{2.2}
    \left|B_r(P) \cap T_j\right| \geq C_1 r^3,
\end{equation}
for every $j=1,\ldots,N$, every $P \in T_j$, and $r \leq r_1$.
\end{assumption}

We show an illustration of a typical model and the assumptions
pertaining to the mesh in Figure~\ref{fig:1}.

We introduce a finite set of numbers,
\[
   \mathcal{Q} = \{\tilde{q}_1,\ldots,\tilde{q}_L\}
\]
representing the possible values which the wavespeed can attain in the
domain $\om$,
\begin{equation}\label{1.2}
   Q_0 = \max\{|\tilde{q}_j|\,:\,j=1,\ldots,L\},
\end{equation}
and
\begin{equation}\label{1.3}
   c_0 = \min\left\{|\tilde{q}_j-\tilde{q}_k|\,:\,
               j,k=1,\ldots,L,\,j\neq k\right\}.
\end{equation}

\begin{assumption}\label{Ass:2}
The potentials are piecewise constant and of the form
\begin{equation}\label{3.1}
   q(x) = \sum_{j=1}^N q_j \chi_{T_j}(x)
\end{equation}
such that $\{T_j\}_{j=1}^N$ is a regular partition of $\om$ with
\begin{equation}\label{3.2}
   N \leq N_0
\end{equation}
for some $N_0$,
\begin{equation}\label{3.3}
   q_j \in \mathcal{Q} \text{ for every }j=1,\ldots,N,
\end{equation}
and
\begin{equation}\label{3.4}
   q_j \neq q_k \text{ if }T_j\text{ is adjacent to }T_k.
\end{equation}
\end{assumption}

We denote by $\|\cdot\|_{\star}$ the norm in
$\mathcal{L}\left(H^{1/2}(\der\om),H^{-1/2}(\der\om)\right)$ defined
by
\[
   \|T\|_{\star} = \sup \{ \langle T\phi, \psi \rangle \,:\,
           \phi, \psi\in H^{1/2}(\der\om), \hbox{ }
   \|\psi\|_{H^{1/2}(\der\om)}=\|\phi\|_{H^{1/2}(\der\om)}=1\}.
\]
We refer to the values of $R$, $r_0$, $K_0$, $r_1$, $Q_0$, $c_0$,
$\omze$, $\omun$ and $N_0$ as to the \textit{a priori data}. 
In the sequel we will introduce a
number of constants that we will always denote by $C$ and, unless
otherwise stated, will depend on a priori data only. The values of
these constants might differ from one line to the other.

We state the main result

\begin{teo}\label{MainTheorem}
Given a domain $\om$ satisfying \eqref{1.1} and \eqref{1.1.5}, a set
of values $\mathcal{Q}$, and $\omega\in[\omze,\omun]$, there exist two
positive constants $\ep_0$ and $C_0$ depending on the a priori data
and on $N_0$ only such that, for every pair of potentials
\begin{equation}\label{5.2}
    q^{(0)} = \sum_{j=1}^Nq_j^{(0)}\chi_{T_j^{(0)}}\text{ and }
    q^{(1)} = \sum_{k=1}^Mq_k^{(1)}\chi_{T_k^{(1)}}
\end{equation}
satisfying Assumptions~\ref{Ass:1} and \ref{Ass:2},
if
\begin{equation}\label{5.4}
   \|\Lambda_{q^{(0)}} - \Lambda_{q^{(1)}}\|_\star \leq \ep_0,
\end{equation}
then
\begin{equation}\label{5.5}
   N = M
\end{equation}
and the order of the tetrahedra can be rearranged so that for every
$j=1,\ldots,N$ we have
\begin{equation}\label{5.6}
   q_j^{(0)} = q_j^{(1)},
\end{equation}
and
\begin{equation}\label{5.7}
   d_\mathcal{H}(T_j^{(0)},T_j^{(1)})
       \leq C_0\| \Lambda_{q^{(0)}}- \Lambda_{q^{(1)}}\|_\star,
\end{equation}
where $d_\mathcal{H}$ denotes the Hausdorff distance.
\end{teo}


\section{A rough stability estimate}
\label{sec:3}

We begin with developing a rough stability estimate for the recovery
of the potential or wavespeed.

\begin{teo}\label{rough}
Given $\om$, $q^{(0)}$, $q^{(1)}$ and $\omega$ as in Theorem
\ref{MainTheorem}, there exist two positive constants $\ep_1<1$ and
$C_2$ depending on $R$, $r_0$, $K_0$, $Q_0$, $\omze$, $\omun$ such
that, for $\| \Lambda_{q^{(0)}}- \Lambda_{q^{(1)}}\|_\star<\ep_1$,
\begin{equation}\label{e1.1}
   \|q^{(0)}-q^{(1)}\|_{L^2(\om)} \leq C_2 \sqrt{N_0} \,
   \left|\log \left(\|\Lambda_{q^{(0)}}
            - \Lambda_{q^{(1)}}\|_\star\right)\right|^{-1/7}.
\end{equation}
\end{teo}

\begin{proof}
We proceed as in \cite{BdHQS}. Alessandrini's identity states that
\begin{equation}\label{e1.2}
   \omega^2 \int_\om (q^{(0)} - q^{(1)}) u_0 u_1 dx
   = \langle (\Lambda_0 - \Lambda_1)(u_0|_{\der\om}),
            u_1|_{\der\om} \rangle
\end{equation}
for every pair of functions $u_0$ and $u_1$ such that
\[
   \Delta u_k + \omega^2q^{(k)}u_k = 0
                \text{ in }\om\text{ for } k=0,1,
\]
where we use the shorthand notation, $\Lambda_k =
\Lambda_{q^{(k)}}$.

We fix $\xi \in \RR^3$ and let $\eta_1$ and $\eta_2$ be unit vectors
in $\RR^3$ such that $\{\xi,\eta_2,\eta_2\}$ is an orthogonal set of
vectors. We let $\mu>0$ be a parameter to be chosen later, and set,
for $k=0,1$,
\begin{equation}\label{e2.1}
    \zeta_k=\left\{\begin{array}{ccc}
    (-1)^{k+1}\frac{\mu}{\sqrt{2}}
    \left(\sqrt{1-\frac{|\xi|^2}{2\mu^2}}\,\,\,\eta_1
    +\frac{(-1)^{k}}{\sqrt{2}\mu}\,\xi+i\,\eta_2\right) & \text{ if } & \frac{|\xi|}{\mu\sqrt{2}}<1, 
\\
                     (-1)^{k+1}\frac{\mu}{\sqrt{2}}\left(\frac{(-1)^{k}}{\sqrt{2}\mu}\,\xi+i\sqrt{\frac{|\xi|^2}{2\mu^2}-1}\,\,\,\eta_1+\eta_2\right) & \text{ if } & \frac{|\xi|}{\mu\sqrt{2}}\geq1.
                   \end{array}
    \right.
\end{equation}
As can be easily checked,
\[
   \zeta_0 + \zeta_1 = \xi,
\]
\[
   \zeta_k \cdot \zeta_k = 0\text{ for }k=0,1
\]
and
\begin{equation}\label{e2.2}
    |\zeta_k| = \max\left\{\mu,\frac{|\xi|}{\sqrt{2}}\right\}.
\end{equation}
We use here complex geometrical optics (CGO) solutions of the
Helmholtz equation and, in particular, the estimates in
\cite[Theorem~3.8]{S} which are due to \cite{H}. For $|\zeta_k| \geq
\max\{\omun^2Q_0,1\} =: c_1$, there is a solution $u_k$ of
\[
   \Delta u_k+\omega^2q^{(k)}u_k=0\text{ in }\om
\]
of the form
\begin{equation}\label{e2.3}
   u_k(x) = e^{i x \cdot \zeta_k} (1+\varphi_k(x)),
\end{equation}
with
\begin{eqnarray}\label{e3.1}
   \|\varphi_k\|_{L^2(\om)} &\leq&
       \frac{C\omun^2Q_0}{|\zeta_k|}\leq\frac{C\omun^2Q_0}{\mu}, 
\nonumber\\
&&
\\
   \|\nabla\varphi_k\|_{L^2(\om)} &\leq& C\omun^2Q_0, 
\nonumber
\end{eqnarray}
where $C=C(R)$.

Inserting \eqref{e2.3} into \eqref{e1.2}, we get
\begin{eqnarray*}
&& \omega^2 \left|(\widehat{q}^{(0)} - \widehat{q}^{(1)})(\xi)
    \right| \leq \left|\langle
    (\Lambda_0-\Lambda_1)(u_0|_{\der\om}),
        u_1|_{\der\om} \rangle\right|
\\
&& \quad + \omega^2 \left|\int_\om (q^{(0)}(x)-q^{(1)}(x))
          e^{i \xi \cdot x}(\varphi_0(x) + \varphi_1(x)
        + \varphi_0(x) \varphi_1(x)) dx \right|
\\
&& \leq \|\Lambda_0-\Lambda_1\|_\star \|u_0\|_{H^1(\om)}
             \|u_1\|_{H^1(\om)}
   + 2\omega^2 Q_0 \left|\int_\om
        (\varphi_0 + \varphi_1 + \varphi_0 \varphi_1) dx\right|.
\end{eqnarray*}
Hence,
\begin{eqnarray*}
&& \left|(\widehat{q}^{(0)} - \widehat{q}^{(1)})(\xi)\right|^2
\\
&& \leq \frac{2}{\omze^4} \|\Lambda_0-\Lambda_1\|^2_\star
   \|u_0\|^2_{H^1(\om)} \|u_1\|^2_{H^1(\om)}
   + 8Q^2_0 \left|\int_\om (\varphi_0 + \varphi_1
          + \varphi_0 \varphi_1) dx\right|^2
\\
&& \leq \frac{2}{\omze^4} \|\Lambda_0-\Lambda_1\|^2_\star
   \|u_0\|^2_{H^1(\om)} \|u_1\|^2_{H^1(\om)}
   + 8Q^2_0 |\om| \left(\|\varphi_0\|_{L^2(\om)} +
     \|\varphi_1\|_{L^2(\om)}\right)
\\[0.2cm]
&& \quad + 8B_0^2 \|\varphi_0\|_{L^2(\om)}
              \|\varphi_1\|_{L^2(\om)}.
\end{eqnarray*}
With \eqref{e2.3} and \eqref{e3.1} we find that there exists a
constant $c_2$ depending only on $R$ such that, for $\mu > c_2$,
\begin{equation}\label{e4.05}
   \|u_k\|_{H^1(\om)} \leq C e^{2R(\mu+|\xi|)},
\end{equation}
$k=0,1$, where $C=C(R,\omun,Q_0)$. Hence,
\begin{equation}\label{e4.1}
   \left|(\widehat{q}^{(0)} - \widehat{q}^{(1)})(\xi)\right|^2
   \leq C \left(e^{8R(\mu+|\xi|)} \|\Lambda_0-\Lambda_1\|^2_\star
                + \frac{1}{\mu^2}\right),
\end{equation}
where $C=C(R,\omze,\omun,Q_0)$. But then, for $\mu \geq
\max(c_1,c_2)$,
\begin{eqnarray}\label{e4.2}
   \|q^{(0)}-q^{(1)}\|^2_{L^2(\om)} &=&
   \int_{|\xi|\leq\rho} \left|
        (\widehat{q}^{(0)}-\widehat{q}^{(1)})(\xi)\right|^2 d\xi
   + \int_{|\xi|>\rho}
     \left|(\widehat{q}^{(0)}-\widehat{q}^{(1)})(\xi)\right|^2 d\xi
\nonumber\\
   &\leq& C \rho^3 \left(e^{8R(\mu+\rho)}
         \|\Lambda_0-\Lambda_1\|^{2}_\star+\frac{1}{\mu^2}\right)
\\
&& + \int_{|\xi|>\rho}
   \left|(\widehat{q}^{(0)}-\widehat{q}^{(1)})(\xi)\right|^2d\xi.
\nonumber
\end{eqnarray}
To estimate the integral in \eqref{e4.2} we show that for every
$s\in(0,1/2)$
\begin{equation}\label{e5.1}
   \|q^{(0)}-q^{(1)}\|^2_{H^{s}(\om)} \leq C\sqrt{N_0},
\end{equation}
where $C=C(R,r_0,Q_0)$. Indeed, by \cite{MP} we have
\begin{eqnarray*}
&&\|q^{(0)}-q^{(1)}\|^2_{H^{s}(\om)} \leq 2\left( \|q^{(0)}\|^2_{H^{s}(\om)}+ \|q^{(1)}\|^2_{H^{s}(\om)}\right)
\\
&& \leq 2\left(\sum_{j=1}^N|q_j^{(0)}|^2|T_j^{(0)}|^{1-2s}|\der T_j^{(0)}|^{2s}+\sum_{k=1}^M|q_k^{(1)}|^2|T_k^{(1)}|^{1-2s}|\der T_k^{(1)}|^{2s} \right)
\\
&& \leq CN_0,
\end{eqnarray*}
where $C=C(R,r_0,Q_0)$. 

Using \eqref{e5.1},
\begin{eqnarray}\label{e5.2}
\int_{|\xi|>\rho} \left|\left(\widehat{q}^{(0)}-\widehat{q}^{(1)}\right)(\xi)\right|^2d\xi  &\leq& \frac{1}{\rho^{2s}}\int_{|\xi|>\rho}
\left(1+|\xi|^{s}\right)^2\left|\left(\widehat{q}^{(0)}-\widehat{q}^{(1)}\right)(\xi)\right|^2d\xi
\nonumber\\
   &\leq&  \frac{1}{\rho^{2s}}\|q^{(0)}-q^{(1)}\|^2_{H^{s}(\om)} \leq \frac{CN_0}{\rho^{2s}}.
\end{eqnarray}
Finally, by inserting \eqref{e5.2} into \eqref{e4.2}, we get that
\begin{equation*}
     \|q^{(0)}-q^{(1)}\|^2_{L^2(\om)}\leq CN_0\left\{\rho^3\left(e^{8R(\mu+\rho)}\|\Lambda_0-\Lambda_1\|^{2}_\star+
     \frac{1}{\mu^2}\right)+\frac{1}{\rho^{2s}}\right\},
\end{equation*}
where $C=C(R,r_0,\omze,\omun,Q_0)$. We then choose
\[
   \rho = \mu^{\frac{2}{3+2s}},
\]
and observe that there is a constant $c_3$ depending only on $R$ such
that, for $\mu\geq c_3$,
\[
   \rho^3 e^{8 R (\mu+\rho)} \leq e^{18 R \mu}
\]
so that
\begin{equation*}
   \|q^{(0)}-q^{(1)}\|^2_{L^2(\om)}
   \leq C N_0 \left(e^{18 R \mu}
     \|\Lambda_0-\Lambda_1\|^{2}_\star
              + \frac{1}{\mu^{\frac{4s}{3+2s}}}\right),
\end{equation*}
where $C=C(R,r_0,\omze,\omun,Q_0)$.

We now take
\[
   \mu = \frac{1}{18R}
        \left|\log\|\Lambda_0-\Lambda_1\|_\star\right|
\]
and assume that
\[
   \|\Lambda_0-\Lambda_1\|_\star \leq e^{-18Rc_3} =: \ep_1
\]
so that $\mu \geq \max\{c_1,c_2,c_3\}$. Then
\begin{equation*}
   \|q^{(0)}-q^{(1)}\|^2_{L^2(\om)}
   \leq C N_0 \left(\|\Lambda_0-\Lambda_1\|^{2}_\star
       + \left|\vphantom{\int}
   \log\|\Lambda_0-\Lambda_1\|_\star\right|^{-\alpha}\right),
\end{equation*}
where $\alpha=\frac{2s}{3+2s}$. The claim follows upon choosing
$s=\frac{1}{4}$.
\end{proof}

Next, we establish an estimate for the Haussdorff distance between two
domain partitions in terms of the difference of potentials defined on
these partitions.

\begin{prop}\label{proptetr}
Given $\om$, $q^{(0)}$ and $q^{(1)}$ as in Theorem~\ref{MainTheorem},
there exists a positive constant $\epsi_1$ depending on $R$, $r_1$,
$Q_0$ and $c_0$ such that, if
\begin{equation}\label{P1.1}
   \|q^{(0)} - q^{(1)}\|_{L^2(\om)} \leq \epsi_1
\end{equation}
then
\begin{equation}\label{P1.2}
   N = M
\end{equation}
and the order of the tetrahedra can be rearranged so that for every
$j=1,\ldots,N$
\begin{equation}\label{P1.3}
   q_j^{(0)} = q_{j}^{(1)}
\end{equation}
and
\begin{equation}\label{P1.4}
   d_\mathcal{H}(T_j^{(0)},T_j^{(1)})
        \leq \frac{\|q^{(0)}-q^{(1)}\|^{2/3}_{L^2(\om)}}{
             \left(c_0^{2}C_1\right)^{1/3}},
\end{equation}
where $c_0$ is given by \eqref{1.3} and $C_1$ by \eqref{2.2}.
\end{prop}

\begin{proof}
We write
\begin{equation}\label{P2.0}
   \epsi = \|q^{(0)}-q^{(1)}\|_{L^2(\om)}.
\end{equation}
For every $l\in\{1,\ldots,L\}$ we let
\begin{equation}\label{P2.1}
   \mathcal{B}_l^{(0)}
   = \left\{j\in\{1,\ldots,N\}\,:\,q_j^{(0)}=\tilde{q}_l\right\}
\end{equation}
and
\begin{equation}\label{P2.2}
   \mathcal{B}_l^{(1)}
   = \left\{k\in\{1,\ldots,M\}\,:\,q_k^{(1)}=\tilde{q}_l\right\}.
\end{equation}
We note that
\begin{eqnarray}\label{P2.3}
  \|q^{(0)}-q^{(1)}\|^2_{L^2(\om)} &=&
  \sum_{l=1}^L\left(\sum_{j\in \mathcal{B}_l^{(0)}}\sum_{k\notin \mathcal{B}_l^{(1)}}\left|q^{(0)}_j-q^{(1)}_k\right|^2\left|T_j^{(0)}\cap T_k^{(1)}\right| \right).
\end{eqnarray}
If $j \in \mathcal{B}_l^{(0)}$ and $k \notin \mathcal{B}_l^{(1)}$
then, by \eqref{1.3},
\[
   \left|q^{(0)}_j-q^{(1)}_k\right|\geq c_0;
\]
hence, by \eqref{P2.3} and \eqref{P2.0}, we have
\begin{eqnarray}
  \epsi^2 &\geq&
   c_0^2 \sum_{l=1}^L\sum_{j\in \mathcal{B}_l^{(0)}}
    \sum_{k \notin \mathcal{B}_l^{(1)}}
            \left|T_j^{(0)}\cap T_k^{(1)}\right|
\end{eqnarray}
so that
\begin{equation}\label{P3.2}
   \left|T_j^{(0)} \cap T_k^{(1)}\right|
   \leq \frac{\epsi^2}{c_0^2} \text{ for every }j,k
        \text{ such that }q_j^{(0)} \neq q_k^{(1)}.
\end{equation}

By assumption \eqref{2.1}, estimate \eqref{P3.2} implies that
$T_j^{(0)} \cap T_k^{(1)}$ is close to $\der T_j^{(0)}$. To make this
precise, we introduce
\[
   T_{j,\delta}^{(0)}
   = \left\{x\in T_j^{(0)} \,:\,
              d(x,\der T_j^{(0)})>\delta\right\}
\]
and prove that
\begin{equation}\label{P3.3}
   T_k^{(1)} \cap T_{j,\delta_\epsi}^{(0)} = \emptyset
\end{equation}
with
\begin{equation}\label{P3.4}
   \delta_\epsi = \left(\frac{\epsi^2}{c_0^2C_1}\right)^{1/3}.
\end{equation}
Indeed, assume that $j \in \mathcal{B}_l^{(0)}$ for some $l \in
\{1,\ldots,N\}$, $k \notin \mathcal{B}_l^{(1)}$ and that there is a
point $P \in T_j^{(0)}\cap T_k^{(1)}$ such that
\begin{equation}\label{P4.1}
   d(P,\der T_j^{(0)})\geq \delta,
\end{equation}
that is, $B_\delta(P) \subset T_j^{(0)}$. Using assumption \eqref{2.1}
and \eqref{2.2} in Remark \ref{rem1}, it then follows that
\begin{equation}\label{P4.3}
   \left|T_j^{(0)} \cap T_k^{(1)}\right|
   \geq \left|B_\delta(P) \cap T_k^{(1)}\right| \geq C_1\delta^3
\end{equation}
if $\delta<r_1$. By \eqref{P3.2}
\begin{equation}\label{P4.4}
    C_1 \delta^3 \leq \frac{\epsi^2}{c_0^2}.
\end{equation}
Thus \eqref{P3.3} holds provided that
\[
   \delta_\epsi
       = \left(\frac{\epsi^2}{c_0^2C_1}\right)^{1/3} \leq r_1,
\]
that is,
\begin{equation}\label{P4.5}
   \epsi \leq \epsi_1=\sqrt{r_1^3c_0^2C_1}.
\end{equation}

Now we consider $T_{j,\delta_\epsi}^{(0)}$ for $\epsi\leq\epsi_1$ and
$j\in \mathcal{B}_l^{(0)}$ for some $l$. Since $\{T_k^{(1)}\}_k$ is a
partition of $\om$, we can write
\begin{eqnarray*}
  T_{j,\delta_\epsi}^{(0)}&=& T_{j,\delta_\epsi}^{(0)}\cap \left(\bigcup_{k=1}^MT_k^{(1)}\right) \\
   &=&\bigcup_{k=1}^M\left(T_{j,\delta_\epsi}^{(0)}\cap T_k^{(1)}\right).
\end{eqnarray*}
Using \eqref{P3.3},
\[
   T_{j,\delta_\epsi}^{(0)} \cap T_k^{(1)}
   = \emptyset\text{ for }k\notin \mathcal{B}_l^{(1)},
\]
and we then obtain
\begin{equation}\label{P5.1}
   T_{j,\delta_\epsi}^{(0)} = \bigcup_{k \in \mathcal{B}_l^{(1)}}
      \left(T_{j,\delta_\epsi}^{(0)}\cap T_k^{(1)}\right).
\end{equation}
If $k_1$ and $k_2\in \mathcal{B}_l^{(1)}$, then $T_{k_1}^{(1)}$ and
$T_{k_2}^{(1)}$ cannot be adjacent by assumption \eqref{3.4}. This
means that there is a unique $k \in \mathcal{B}_l^{(1)}$ such that
\begin{equation}\label{P5.2}
    T_{j,\delta_\epsi}^{(0)} \cap T_k^{(1)} \neq \emptyset
\end{equation}
and, with \eqref{P5.1},
\[
   T_{j,\delta_\epsi}^{(0)}
   = T_{j,\delta_\epsi}^{(0)} \cap T_k^{(1)}\subset T_k^{(1)}.
\]
Thus we proved that for every $j \in \{1,\ldots,N\}$ there is a unique
index $\overline{k}(j) \in \{1,\ldots,M\}$ such that
\begin{equation}\label{P6.1}
   q^{(0)}_j = q^{(1)}_{\overline{k}(j)}
\end{equation}
and
\begin{equation}\label{P6.2}
   T_{j,\delta_\epsi}^{(0)} \subset T_{\overline{k}(j)}^{(1)}.
\end{equation}
In particular, this implies that $M \geq N$.

By interchanging the roles of $q^{(0)}$ and $q^{(1)}$ it follows that
$M=N$, $\overline{k}$ is a permutation on $\{1,\ldots,N\}$ and
\[
   T_{j,\delta_\epsi}^{(0)} \subset T_{\overline{k}(j)}^{(1)}
\text{ and }T_{\overline{k}(j),\delta_\epsi}^{(1)}\subset T_{j}^{(0)}
\]
that, by \eqref{P3.4}, gives \eqref{P1.4}.
\end{proof}

Combining Theorem~\ref{rough} and Proposition~\ref{proptetr}, we
obtain the following logarithmic stability estimate

\begin{coro}\label{cor1}
Under the assumptions of Theorem \ref{rough}, there is a constant
$\ep_2 < 1$ depending only on the a priori data such that, if
\[
   \|\Lambda_{q^{(0)}} - \Lambda_{q^{(1)}}\|_\star\leq\ep_2
\]
then
\[
   N = M
\]
and the order of tetrahedra can be rearranged so that
\[
   q_j^{(0)} = q_j^{(1)}
\]
and
\begin{equation}\label{corod}
   d_\mathcal{H}(T_j^{(0)},T_j^{(1)}) \leq
   \left(\frac{C_2^2 N_0}{c_0^2C_1}\right)^{1/3}
   \left|\log \left(\| \Lambda_{q^{(0)}}
         - \Lambda_{q^{(1)}}\|_\star\right)\right|^{-2/21}.
\end{equation}
\end{coro}

\section{Geometric estimates, construction of an intermediate
         partition and augmenting the domain}
\label{sec:4}

Here, we map the information on the Haussdorff distance of tetrahedra
in information on the distance between vertices of these tetrahedra.
It is straightforward to see that if $T^{(k)}$, $k=0,1$, are
tetrahedra generated by vertices $P_i^{(k)}$, $i=1,2,3,4$, that then
\begin{equation}\label{geo1}
   d_\mathcal{H}(T^{(0)},T^{(1)})
   \leq \min_\wp \max_{1 \leq i \leq 4}
           \left|P_i^{(0)}-P_{\wp(i)}^{(1)}\right|,
\end{equation}
where $\wp$ denotes a permutation on the set $\{1, 2, 3, 4\}$.
Moreover, if $T^{(k)}\subset B_R(0)$ and satisfies assumption
\eqref{1.5.5} for $k=0,1$, then there exists a positive constant
$A_1$, depending on $R$ and $r_1$ only, such that
\begin{equation}\label{geo2}
    \min_\wp \max_{1 \leq i \leq 4}
          \left|P_i^{(0)}-P_{\wp(i)}^{(1)}\right|
    \leq A_1 d_\mathcal{H}(T^{(0)},T^{(1)}).
\end{equation}
Using Corollary \ref{cor1} we then obtain

\begin{prop}\label{vic}
Under the assumptions of Theorem \ref{rough}, there is a positive
constant $\ep_3<1$ such that if
\[
   \|\Lambda_{q^{(0)}} - \Lambda_{q^{(1)}}\|_\star \leq \ep_3
\]
then for every vertex $P^{(0)}_{j,i}$ of $T_j^{(0)}$ (with
$i=1,2,3,4$) there is a unique vertex $P^{(1)}_{j,i}$ of $T_j^{(1)}$
such that
\begin{equation}\label{P7.2}
   d(P^{(0)}_{j,i},P^{(1)}_{j,i}) \leq \frac{d_1}{4}
\end{equation}
for $d_1$ as in \ref{2.1}.
\end{prop}

\begin{proof}
It is sufficient to consider $\ep_3<1$, such that
\[
   A_1 \left(\frac{C_2^2 N_0}{c_0^2C_1}\right)^{1/3}
       \left|\log \left(\ep_3\right)\right|^{-2/21}
              < \frac{d_1}{4},
\]
and the statement follows.
\end{proof}

We introduce a deformation of the tetrahedra forming the partition
of $\om$. To this end, for each $j \in \{1,\ldots,N\}$, we define
tetrahedra $T_j^{(t)}$ by its vertices,
\begin{equation}\label{vertici}
   P_{j,i}^{(t)} = P_{j,i}^{(0)} + t v_{j,i}\text{ for } t \in [0,1],
\end{equation}
where
\begin{equation}\label{v}
   v_{j,i} = P_{j,i}^{(1)} - P_{j,i}^{(0)}.
\end{equation}
The resulting partition $\{ T_j^{(t)} \}_j$ is a regular partition of
$\om$ satisfying condition \eqref{1.5.5}. We point out that, by
\eqref{geo1} and \eqref{geo2}, there is a positive constant $A_2>1$
such that
\begin{equation}\label{s1}
   A_2^{-1}\left(\sum_{i=1}^4|v_{j,i}|^2\right)^{1/2}
     \leq d_\mathcal{H}\left(T_j^{(0)},T_j^{(1)}\right)
             \leq A_2\left(\sum_{i=1}^4|v_{j,i}|^2\right)^{1/2}.
\end{equation}

We define
\[
   q^{(t)}=\sum_{j=1}^Nq_j\chi_{T_j^{(t)}},
\]
where we denoted by $q_j=q_j^{(0)}=q_j^{(1)}$. A suggestion of
Alessandrini ( \cite{A-personal-communication}) allows us to avoid the
assumption thet $q$ is known on $\partial\Omega$.  To this aim we
extend our domain and introduce a regular domain $\tilde{\om}$
containing $\om$; we extend each potential $q^{(t)}$, for $t \in
[0,1]$, to $\tilde{\om}$ with the same constant value,
$\tilde{q}_0$. The particular choice of value $\tilde{q}_0$ for this
extension does not matter, as long as we are able to ensure
well-posedness of the corresponding Dirichlet problem. For this reason
we choose a special value. We take $\tilde{R} = \frac{2}{\sqrt{3}} R$,
so that
\begin{equation}\label{a1}
   \lambda_1(B_{\tilde{R}}) = \frac{3}{4}\lambda_1(B_R),
\end{equation}
and choose
\begin{equation}\label{omegatilde}
   \tilde{\om} = B_{\tilde{R}}(0).
\end{equation}
We then define
\begin{equation}\label{a2}
   \tilde{q}^{(t)} = \tilde{q}_0
         + (q^{(t)}-\tilde{q}_0) \chi_\om\text{ for }t \in [0,1],
\end{equation}
with $\tilde{q}_0 = Q_0$ (cf.~\eqref{1.2}). For $\omega \leq \omun$
and $t \in [0,1]$, we have
\[
   \left|\omega^2\tilde{q}^{(t)}\right|
   \leq \omun^2 Q_0
   \leq \frac{1}{2} \, \lambda_1(B_R)
   = \frac{2}{3}\,\lambda_1(\omti),
\]
cf.~\eqref{a1} and \eqref{3.5}, whence the Dirichlet problem
\begin{equation}\label{a3}
\left\{\begin{array}{rcl}
   \Delta u+\omega^2 \tilde{q}^{(t)} u &=& 0\text{ in }\omti,
\\
   u &=& \phi \text{ on }\der \omti,
\end{array}\right.
\end{equation}
has a unique solution $u \in H^{1}(\omti)$ for every $\phi \in
H^{1/2}(\der \omti)$. Thus the one-parameter family of
Dirichlet-to-Neumann maps,
\begin{equation}\label{Lambdat}
   \tilde{\Lambda}_t = \Lambda_{\tilde{q}^{(t)}},
   \text{ for }t \in [0,1]
\end{equation}
is well defined in $\mathcal{L}(H^{1/2}(\der \omti),H^{-1/2}(\der
\omti))$. We denote the norm in this space by $\|T\|_{\tilde{\star}}$.

To proceed, we take $\phi, \psi \in H^{1/2}(\der\omti)$ and let
$\tilde{u}_0$ and $\tilde{u}_1$ be the solutions to
\begin{equation*}
    \left\{\begin{array}{rcl}
             \Delta \tilde{u}_0+\omega^2 \tilde{q}^{(0)} \tilde{u}_0 & = & 0\text{ in }\omti,\\
             \tilde{u}_0 & = & \phi \text{ on }\der \omti,
           \end{array}
               \right. \text{ and }
 \left\{\begin{array}{rcl}
             \Delta \tilde{u}_1+\omega^2 \tilde{q}^{(1)}\tilde{u}_1 & = & 0\text{ in }\omti,\\
             \tilde{u}_1 & = & \psi \text{ on }\der \omti.
           \end{array}
    \right.
\end{equation*}
We then use Alessandrini's identity and write
\begin{eqnarray*}
&& \langle(\tilde{\Lambda}_1 - \tilde{\Lambda}_0)(\phi),\psi\rangle
   = \int_{\omti} (\tilde{q}^{(1)} - \tilde{q}^{(0)})
                \tilde{u}_0 \tilde{u}_1 dx
   = \int_{\om} (q^{(1)} - q^{(0)}) \tilde{u}_0 \tilde{u}_1 dx
\\
&& = \langle(\Lambda_1-\Lambda_0)(\tilde{u}_0|_{\der\om}),
        {\tilde{u}_1}|_{\der\om}\rangle
   \leq \|\Lambda_1-\Lambda_0\|_\star
        \|\tilde{u}_0\|_{H^{1/2}(\der\om)}
            \|\tilde{u}_1\|_{H^{1/2}(\der\om)}.
\end{eqnarray*}
Moreover, by trace and regularity estimates, we have
\begin{equation*}
   \|\tilde{u}_k\|_{H^{1/2}(\der\om)}
   \leq C \|\tilde{u}_k\|_{H^{1}(\omti)}
   \leq C \|\tilde{u}_k\|_{H^{1/2}(\der\omti)}\text{ for }k=0,1,
\end{equation*}
where $C$ depends on the a priori data. We have then shown that
\begin{equation}\label{a5}
   \|\tilde{\Lambda}_1-\tilde{\Lambda}_0\|_{\tilde{\star}}
               \leq C_3 \|\Lambda_1-\Lambda_0\|_\star.
\end{equation}

\section{Proof of Lipschitz stability}
\label{sec:5}

In this section, we give the proof of Lipschitz stability starting
from the logarithmic estimate obtained in Corollary \ref{cor1}. We
split the proof into three steps:
\begin{description}
\item[First step.] We show that for any pair of functions $\phi$ and
  $\psi$ in $H^{1/2}(\der\omti)$, the function
\[
   \mathcal{F}(t,\phi,\psi)
         = \langle \tilde{\Lambda}_t(\phi),\psi \rangle
\]
is differentiable.
\item[Second step.] We show that there is a positive constant $L_1$
  and a number $\alpha\in(0,1)$ depending on the a-priori data such that for any $\phi$ and $\psi$ in
  $H^{1/2}(\der\omti)$,
\begin{equation}\label{secondstep}
   \left|\frac{d}{dt} \mathcal{F}(t,\phi,\psi)
      -\frac{d}{dt}\mathcal{F}(t,\phi,\psi)_{|_{t=0}}\right|
   \leq L_1 d_T^{1+\alpha}
      \|\phi\|_{H^{1/2}(\der\omti)}\|\psi\|_{H^{1/2}(\der\omti)}.
\end{equation}
\item[Third step.] Finally, we prove that there is a positive constant
  $m_1$ such that, for special choices of non-zero functions $\phi_0$
  and $\psi_0$, we have
\begin{equation}\label{thirdstep}
   \left|\frac{d}{dt}\mathcal{F}(t,\phi_0,\psi_0)_{|_{t=0}}\right|
   \geq m_1 d_T\|\phi_0\|_{H^{1/2}(\der\omti)}
                    \|\psi_0\|_{H^{1/2}(\der\omti)}.
\end{equation}
\end{description}
Here, $d_T = \sum_{j=1}^N
d_\mathcal{H}\left(T_j^{(0)},T_j^{(1)}\right)$.

Once these three steps have been proven we conclude that
\begin{eqnarray*}
&& \left| \langle(\tilde{\Lambda}_1
      - \tilde{\Lambda}_0)(\phi_0),\psi_0 \rangle\right|
   = \left|\mathcal{F}(1,\phi_0,\psi_0)
      - \mathcal{F}(0,\phi_0,\psi_0)\right|
   = \left| \int_0^1
     \frac{d}{dt}\mathcal{F}(t,\phi_0,\psi_0)\right|
\\
&& \geq \left|
      \frac{d}{dt}\mathcal{F}(t,\phi_0,\psi_0)_{|_{t=0}}\right|
   - \int_0^1 \left| \frac{d}{dt} \mathcal{F}(t,\phi_0,\psi_0)
     - \frac{d}{dt}\mathcal{F}(t,\phi_0,\psi_0)_{|_{t=0}}\right|
\\[0.2cm]
&& \geq \|\phi_0\|_{H^{1/2}} \|\psi_0\|_{H^{1/2}} d_T
                     \left(m_1- L_1d_T^{\alpha}\right),
\end{eqnarray*}
that is,
\begin{equation}\label{s2}
   \|\tilde{\Lambda}_1-\tilde{\Lambda}_0\|_\star
           \geq d_T\left(m_1 - L_1d_T^{\alpha}\right).
\end{equation}
By Corollary \ref{cor1}, there exists a positive constant $\ep_0 \leq
\ep_3$ such that, if
\[
   \|\Lambda_1-\Lambda_0\|_\star \leq \ep_0
\]
then
\[
   \left(m_1 - L_1 d_T^{\alpha}\right) \geq \frac{m_1}{2}
\]
and, hence, by \eqref{a5}
\[
   d_T \leq \frac{m_1}{2}
     \|\tilde{\Lambda}_1-\tilde{\Lambda}_0\|_{\tilde{\star}}
   \leq \frac{m_1 C_3}{2} \|\Lambda_1-\Lambda_0\|_{\star},
\]
which implies \eqref{5.7}.

\subsection{First step: Differentiability of
                        $\mathcal{F}(t,\phi,\psi)$}

Let $\phi, \psi \in H^{1/2}(\der\omti)$ and let $t_0 \in [0,1]$. For
$h \neq 0$ such that $t_0 + h \in[0,1]$ we introduce the finite
difference
\begin{equation}\label{d1}
   R(h) = \frac{1}{h}
   \left(\mathcal{F}(t_0+h,\phi,\psi)-\mathcal{F}(t_0,\phi,\psi)
         \right).
\end{equation}
For $t \in [0,1]$ fixed, we let $u(x;t)$ and $v(x;t)$ be the (unique)
solutions in $H^1(\omti)$ to the boundary value problems,
\[
   \left\{\begin{array}{rcl}
   \Delta u(x;t)+\omega^2\tilde{q}^{(t)}(x)u(x;t)
         &=& 0\text{ for }x\in\omti,
\\
   u(x;t)&=&\phi(x)\text{ for }x\in\der\omti
   \end{array}\right.
\]
and
\[
  \left\{\begin{array}{rcl}
  \Delta v(x;t)+\omega^2\tilde{q}^{(t)}(x)v(x;t)
        &=& 0\text{ for }x\in\omti,
\\
  v(x;t)&=&\psi(x)\text{ for }x\in\der\omti.
  \end{array}\right.
\]
Applying Alessandrini's identity and the definition of
$\tilde{q}^{(t)}$, we find that
\begin{eqnarray*}
   R(h) &=&\frac{\omega^2}{h}
   \int_\om \left(q^{(t_0+h)}(x)-q^{(t_0)}(x)\right)
            u(x;t_0+h)v(x;t_0) dx
\\
        &=& \frac{\omega^2}{h} \sum_{j=1}^N q_j
   \left\{\int_{T_j^{(t_0+h)}} u(x;t_0+h)v(x;t_0)dx-
      \int_{T_j^{(t_0)}} u(x;t_0+h)v(x;t_0)dx\right\}.
\end{eqnarray*}

For any index $j\in\{1,\ldots,N\}$ we define $\Phi_{j,t_0} :\ \RR^3
\to \RR^3$ as the affine map with the property that
\begin{equation}\label{Phi}
   \Phi_{j,t_0}(P^{(0)}_{j,i} + t_0 v_{j,i}) = v_{j,i}\text{ for
   }i=1,2,3,4,
\end{equation} 
where $P^{(0)}_{j,i}$ is defined in \eqref{vertici} and $v_{j,i}$ in
\eqref{v}. We let
\begin{equation}\label{Fj}
   F^{t_0}_{j,\tau}(x) = x + \tau \Phi_{j,t_0}(x)
\end{equation}
so that $F^{t_0}_{j,\tau}(T_j^{(t_0)}) = T_j^{(t_0+\tau)}$. We note
that with assumption \eqref{2.1}
\begin{equation}\label{stimePhi}
   \left|\Phi_{j,t_0}\right|
   + \left|\operatorname{div}\Phi_{j,t_0}\right| \leq C(R,r_1).
\end{equation}

By using $F^{t_0}_{j,h}$ as a change of variable, we get
\begin{equation}\label{R2}
   R(h) = \frac{\omega^2}{h} \sum_{j=1}^N q_j
                  \int_{T_j^{(t_0)}} \mu_j(x,t_0)dx,
\end{equation}
where
\begin{equation}\label{mu}
   \mu_j(x,t_0) = u(F^{t_0}_{j,h}(x);t_0+h)
            v(F^{t_0}_{j,h}(x);t_0)
   |\det DF^{t_0}_{j,h}(x)| - u(x;t_0+h)v(x;t_0).
\end{equation}
We proceed with the analysis on each tetrahedron $T_j^{(t_0)}$ in the
same way and for simplicity of notation drop the index $j$.

By standard regularity estimates for solutions of elliptic equations,
we know that $u(\cdot,t)$ and $v(\cdot,t)$ belong to
$C^{1,\alpha}\left(\om\right)$ for some $\alpha\in(0,1)$ and that
\begin{equation}\label{uC1alpha}
   \|u(\cdot;t)\|_{C^{1,\alpha}\left(\om\right)}
                    \leq C \|\phi\|_{H^{1/2}(\der\omti)}, 
\end{equation}
\begin{equation}\label{vC1alpha}
   \|v(\cdot,t)\|_{C^{1,\alpha}\left(\om\right)}
                    \leq C \|\psi\|_{H^{1/2}(\der\omti)},
\end{equation}
where $C$ depends on the a priori data. Thus,
\begin{equation}\label{mu1}
   u(F^{t_0}_{h}(x);t_0+h) - u(x;t_0+h)
      = h \nabla u(x;t_0+h) \cdot \Phi_{t_0}(x)+\eta_1(h).
\end{equation}
For some $\xi$ between $x$ and $F^{t_0}_h(x)=x+h\Phi_{t_0}(x)$,
\begin{eqnarray}\label{mu2}
|\eta_1(h)|&=&\left|h\nabla u(\xi;t_0+h)\cdot \Phi_{t_0}(x)-h\nabla u(x;t_0+h)\cdot \Phi_{t_0}(x)\right|
\nonumber\\
& \leq &|h| \|u(\cdot,t_0+h)\|_{C^{1,\alpha}\left(\om\right)}|\xi-x|^\alpha\left|\Phi_{t_0}(x)\right|
\nonumber\\
& \leq & C\|\phi\|_{H^{1/2}(\der\omti)}\left(|h|\right)^{1+\alpha}\left|\Phi_{t_0}(x)\right|
\nonumber\\
& \leq & C\|\phi\|_{H^{1/2}(\der\omti)}|h|^{1+\alpha},
\end{eqnarray}
where we used \eqref{stimePhi} in the last estimate. A similar
estimate holds for $v(F^{t_0}_{h}(x);t_0+h)-v(x;t_0+h)$. Moreover, by
direct calculation,
\begin{equation}\label{mu3}
   \left|\det DF^{t_0}_{h}(x)\right|
   = 1 + h \, \operatorname{div}\left(\Phi_{t_0}\right)
       + o(h).
\end{equation}
Using \eqref{mu}, \eqref{mu1}, \eqref{mu2} and \eqref{mu3}, we get
\begin{equation}\label{mu4}
   \mu(x,t_0)=h \, \operatorname{div}\left(u(x;t_0+h)
        v(x;t_0) \Phi_{t_0}(x)\right) + \eta(h)
\end{equation}
with
\begin{equation}\label{mu5}
   |\eta(h)| \leq C |h|^{1+\alpha},
\end{equation}
where $C$ depends on the a priori data and on
$\|\phi\|_{H^{1/2}(\der\omti)}$ and
$\|\psi\|_{H^{1/2}(\der\omti)}$. By inserting estimates \eqref{mu4}
and \eqref{mu5} into \eqref{R2} we obtain
\begin{equation}
   R(h) = \omega^2 \sum_{j=1}^N q_j \int_{T_j^{(t_0)}}
      \operatorname{div}\left(u(x;t_0+h) v(x;t_0)
                 \Phi_{j,t_0}(x)\right)dx+O(h^\alpha).
\end{equation}

Applying usual energy estimates, we find that
\begin{equation}\label{ee}
   \|u(\cdot,t_0+h)-u(\cdot,t_0)\|_{H^1(\om)}
   \leq C \omega^2 \|q^{(t_0+h)}-q^{(t_0)}\|_{L^2(\om)}
            \|\phi\|_{H^{1/2}(\der\omti)}
\end{equation}
and, hence,
\[
   \lim_{h\to 0}R(h) = \omega^2 \sum_{j=1}^N
         q_j \int_{T_j^{(t_0)}}
   \operatorname{div}\left(u(x;t_0) v(x;t_0)
                        \Phi_{j,t_0}(x)\right) dx.
\]
This implies that $\mathcal{F}(t,\phi,\psi)$ is differentiable and
that
\begin{equation}\label{deriv}
   \frac{d}{dt} \langle \tilde{\Lambda}_t(\phi),\psi \rangle_{t=t_0}
   = \omega^2\sum_{j=1}^N q_j\int_{T_j^{(t_0)}}
   \operatorname{div}\left(u(x;t_0) v(x;t_0)
                        \Phi_{j,t_0}(x)\right) dx.
\end{equation}
Using the divergence theorem, we obtain
\begin{equation}\label{deriv2}
   \frac{d}{dt} \langle \Lambda_t(\phi),\psi \rangle_{t=t_0}
   = \omega^2\sum_{j=1}^N q_j
   \int_{\der T_j^{(t_0)}} u(x;t_0) v(x;t_0)
           \left(\Phi_{j,t_0}(x)\cdot\nu_j\right) d\sigma_x,
\end{equation}
where $\nu_j$ is the exterior normal to $\der T_j^{(t_0)}$ and
$d\sigma_x$ is the surface measure.

\subsection{Second step: Behavior of
      $\frac{d}{dt}\mathcal{F}(t,\phi,\psi)$ with respect to $t$}

In this subsection, we estimate, for any fixed $t \in [0,1]$, the
quantity
\[
   \tilde{J} = \frac{d}{dt}\mathcal{F}(t,\phi,\psi)
         - \frac{d}{dt}\mathcal{F}(t,\phi,\psi)|_{t=0}.
\]
By \eqref{deriv}, we can write
\begin{equation}\label{Jtilde}
   \tilde{J} = \omega^2\sum_{j=1}^N q_j J_j
\end{equation}
where
\[
   J_j = \int_{T_j^{(t)}}
    \operatorname{div}\left(u(x;t)v(x;t)\Phi_{j,t}(x)\right) dx
       - \int_{T_j^{(0)}}
    \operatorname{div}\left(u(x;0)v(x;0)\Phi_{j,0}(x)\right) dx.
\]
We write
\begin{equation}\label{Vj}
   V_j = \sum_{i=1}^4 \left|v_{j,i}\right|.
\end{equation}
Since, here, we focus on each tetrahedron separately, we drop the
index $j$ from $J_j$, $T_j^{(t)}$, $T_j^{(0)}$, $\Phi_{j,t}$,
$\Phi_{j,0}$, and $V_j$, again, for simplicity of notation. We use the
change of variable $F_t(x)=F_{j,t}$ as defined in \eqref{Fj}, and get
\[
   J=\int_{T^{(0)}}\left(\operatorname{div}_y\left(
        u(y;t)v(y;t)\Phi_{t}(y)\right)_{y=F_t(x)}
   \left|\det DF_t(x)\right|
   -\operatorname{div}_x\left(u(x;0)v(x;0)\Phi_{0}(x)\right)
                      \right)dx
\]
We introduce the quantity
\[
   G(y,t) =
      \operatorname{div}_y\left(u(y;t)v(y;t)\Phi_t(y)\right),
\]
and estimate $J$,
\begin{eqnarray*}
   J &=& \left|\int_{T^{(0)}}
         \left(G(F_t(x),t)\left|\det DF_t(x)\right|
                          -G(x,0)\right) dx\right|
\nonumber\\
   &\leq& \int_{T^{(0)}}
    \! \left|G(F_t(x),t)-G(x,0)\right| \left|\det DF_t(x)\right| dx
    + \int_{T^{(0)}} \!
       \left|G(x,0)\right| \left|\det DF_t(x)-1\right| dx
\nonumber\\
   &=& J^{(1)}+J^{(2)},
\end{eqnarray*}
in which
\begin{eqnarray*}
   J^{(1)} &\leq&
   C\left\{\int_{T^{(0)}}\left|\nabla_y \left(u(y;t)v(y;t)\right)_{|_{y=F_t(x)}}-\nabla\left(u(x;0)v(x;0)\right)\right| \left|\Phi_0(x)\right|dx \right.
\\
&&+\left.\int_{T^{(0)}} \left|u(F_t(x);t)v(F_t(x);t)\left(\operatorname{div}\Phi_t(y)\right)_{|_{y=F_t(x)}}- u(x;0)v(x;0)\left(\operatorname{div}\Phi_0(x)\right)\right|dx\right\},
\end{eqnarray*}
using that $\Phi_t(F_t(x)) = \Phi_0(x)$. A straightforward calculation
gives
\[
   \left(\operatorname{div} \Phi_t(y)\right)|_{y=F_t(x)}
   = \operatorname{div} \Phi_0(x) - t \,
     \operatorname{tr}\left(D\Phi_t(F_t(x))D\Phi_0(x)\right).
\]
Hence, writing
\[
   w(y;t) = u(y;t)v(y;t)
\]
we obtain the estimate
\begin{eqnarray*}
J^{(1)}&\leq& C\left\{ \int_{T^{(0)}}\left|\nabla w(F_t(x);t)-\nabla w(x;0)\right|\left|\Phi_0(x)\right|dx\right.
\\
&&\quad+\int_{T^{(0)}} |w(F_t(x);t)-w(x;0)|\left|\operatorname{div}\Phi_0(x)\right|dx
\\
&&\left.\quad\quad+t\int_{T^{(0)}}\left|w(F_t(x);t)\right|
\left|\text{tr}\left(D\Phi_t(F_t(x))D\Phi_0(x)\right)\right|dx\right\}.
\end{eqnarray*}
Using \eqref{Phi} and \eqref{Vj}, we find that
\begin{equation}\label{Phiest}
   \left|\Phi_t(x)\right|
        + \left|D\Phi_t(x)\right| \leq C V
\end{equation}
and, hence,
\begin{eqnarray*}
   J^{(1)} &\leq& C V
   \left\{ \int_{T^{(0)}}
   \left|\nabla w(F_t(x);t) - \nabla w(x;0)\right|
         + \left|w(F_t(x);t) - w(x;0)\right| dx \right\}
\\
&& + C V^2 \|\phi\|_{H^{1/2}(\der\omti)}
                    \|\psi\|_{H^{1/2}(\der\omti)}.
\end{eqnarray*}
We analyze the term containing $\nabla w$. By combining
\eqref{uC1alpha}, \eqref{vC1alpha} and \eqref{ee} and using the fact
that $|F_t(x)-x| = t \, |\Phi_t(x)| \leq CV$, we obtain
\begin{eqnarray*}
&&\int_{T^{(0)}} \left|\nabla w(F_t(x);t)-\nabla w(x;0)\right|dx
\\
&& \quad\leq
    \int_{T^{(0)}} \left(\left|\nabla w(F_t(x);t)-\nabla w(x;t)\right|+\left|\nabla w(x;t)-\nabla w(x;0)\right|\right)dx\\
&& \quad\leq
    C \|\phi\|_{H^{1/2}(\der\omti)}\|\psi\|_{H^{1/2}(\der\omti)}\left(V^\alpha+\omega^2\|q^{(t)}-q^{(0)}\|_{L^2(\om)}\right).
\end{eqnarray*}
Then, by \eqref{1.2}, \eqref{3.5} and \eqref{s1},
\[
   \omega^2 \|q^{(t)}-q^{(0)}\|_{L^2(\om)} \leq C \sum_{j=1}^N V_j
\]
and, so,
\[\int_{T^{(0)}}\left|\nabla w(F_t(x);t)-\nabla w(x;0)\right|dx\leq  C\|\phi\|_{H^{1/2}(\der\omti)}\|\psi\|_{H^{1/2}(\der\omti)}\left(V^\alpha+\sum_{j=1}^NV_j\right).\]
An analogous estimate holds for $\int_{T^{(0)}}\left|w(F_t(x);t)-\nabla w(x;0)\right|dx$.
Finally, by recalling \eqref{Vj}, we obtain
\begin{equation}\label{J1a}
J^{(1)}\leq
C\|\phi\|_{H^{1/2}}\|\psi\|_{H^{1/2}}\left(\sum_{j=1}^NV_j\right)^{1+\alpha}.\end{equation}
The integral, $J^{(2)}$, can be estimated in a similar way by observing
that, by \eqref{uC1alpha}, \eqref{vC1alpha} and \eqref{Phiest},
\begin{equation}\label{J2a}
\left|\operatorname{div}\left(u(x,0)v(x,0)\Phi_0(x)\right)\right|\leq  C\|\phi\|_{H^{1/2}(\der\omti)}\|\psi\|_{H^{1/2}(\der\omti)}V
\end{equation}
and, by \eqref{Fj} and \eqref{Phiest},
\begin{equation}\label{J2b}
   \left|\det DF_t(x)-1\right| \leq CV.
\end{equation}
By combining \eqref{Jtilde}, \eqref{J1a}, \eqref{J2a} and \eqref{J2b}
and adding up the contributions from all the tetrahedra,  we get
\[
   \tilde{J} \leq L_1  \|\phi\|_{H^{1/2}(\der\omti)}
         \|\psi\|_{H^{1/2}(\der\omti)}
       \left(\sum_{j=1}^NV_j\right)^{1+\alpha}
\]
and, by \eqref{Jtilde}, \eqref{1.2}, \eqref{3.5} and \eqref{s1}, we
finally arrive at estimate \eqref{secondstep} and conclude the proof
of second step.

\subsection{Third step: Lower bound of
            $\frac{d}{dt}\mathcal{F}(t,\phi,\psi)|_{t=0}$}

With \eqref{deriv2} the Gateaux derivative is given by
\[
   \frac{d}{dt}\mathcal{F}(t,\phi,\psi)|_{t=0}
   = \omega^2 \sum_{j=1}^N q_j
     \int_{\der T_j^{(0)}} u(x) w(x)
         \left(\Phi_{j,0}(x) \cdot \nu_j\right) d\sigma_x,
\]
where $u$ and $w$ solve problems
\[
  \left\{\begin{array}{rcl}
  \Delta u + \omega^2q^{(0)} u&=&0\text{ in }\omti,
\\
  u&=&\phi\text{ on }\der\omti.\end{array}\right.
\]
and
\[
  \left\{\begin{array}{rcl}
  \Delta w + \omega^2q^{(0)} w&=&0\text{ in }\omti,
\\
  w&=&\psi\text{ on }\der\omti.\end{array}\right.,
\]
respectively. We introduce
\begin{equation}\label{vnorm}
   \tilde{v}_{j,i} = \frac{v_{j,i}}{\sum_{j=1}^N V_j}
   \text{ for }j\in\{1,\ldots,N\},\quad i=1,2,3,4,
\end{equation}
where $V_j$ is defined as in \eqref{Vj}, and note that
\begin{equation}\label{normalizz}
    \sum_{j=1}^N \sum_{i=1}^4 |\tilde{v}_{j,i}| = 1
\end{equation}
We also let
\begin{equation}\label{phinorm}
    \tilde{\Phi}_j(x)=\frac{\Phi_{j,0}(x)}{\sum_{l=1}^NV_l},
\end{equation}
and consider the bilinear operator
\begin{equation}\label{L}
    \mathcal{G}(\phi,\psi) = \sum_{j=1}^N q_j
    \int_{\der T_j^{(0)}} u(x) w(x)
                       (\tilde{\Phi}_j(x)\cdot\nu_j)
\end{equation}
in $\mathcal{L}(H^{1/2}(\der \omti),H^{-1/2}(\der \omti))$. Now, for
every $\phi$ and $\psi$ in $H^{1/2}(\der \omti)$, we have
\begin{equation}\label{s2-5}\left|\mathcal{G}(\phi,\psi)\right|\leq m_0\|\phi\|_{H^{1/2}}\|\psi\|_{H^{1/2}},\end{equation}
where
\begin{equation}\label{m0}
   m_0 = \|\mathcal{G}\|_{\tilde{\star}}.
\end{equation}
We choose boundary values corresponding to CGO solutions: Let $\xi$ be
any vector in $\RR^3$ and let $\mu$ be a positive parameter to be
chosen later, and let $\zeta_0$ and $\zeta_1$ as in \eqref{e2.1}. We
form
\begin{equation}\label{uzerotilde}
    \tilde{u}_0 = e^{i x \cdot \zeta_0}(1 + \varphi_0(x))
\end{equation}
and
\begin{equation}\label{wzerotilde}
    \tilde{w}_0 = e^{i x \cdot \zeta_1}(1 + \varphi_1(x)),
\end{equation}
which are both solutions of the equation $\Delta u + \omega^2q^{(0)}u
= 0$ in $\omti$, such that
\begin{eqnarray}\label{e3.1bis}
  \|\varphi_k\|_{L^2(\omti)} &\leq& \frac{C_0 \omun^2 Q_0}{\mu}, 
\nonumber\\
&& \\
  \|\nabla \varphi_k\|_{L^2(\omti)} &\leq& C_0 \omun^2 Q_0,
\nonumber
\end{eqnarray}
and $\zeta_0+\zeta_1=\xi$. Substituting these functions into
\eqref{s2-5}, by \eqref{1.2} and \eqref{Phiest}, we get
\begin{multline}\label{s1-7}
   \left|\sum_{j=1}^N q_j
      \int_{\der T_j^{(0)}} e^{i x \cdot \xi}
          (\tilde{\Phi}_j(x) \nu_j)\right|
   \leq m_0 \|\phi\|_{H^{1/2}(\der\omti)}
               \|\psi\|_{H^{1/2}(\der\omti)}
\\[-0.2cm]
   + C \sum_{j=1}^N \int_{\der T_j^{(0)}}
     \left(|\varphi_0| + |\varphi_1| + |\varphi_0||\varphi_1|\right).
\end{multline}

We now estimate last term in \eqref{s1-7}. We recall the interpolation
estimate for $0<\tau<1$
\begin{equation}\label{s3-7}
    \|\varphi_k\|_{H^\tau(\omti)}\leq C\|\varphi_k\|^{1-\tau}_{L^2(\omti)}\left(\|\varphi_k\|_{L^2(\omti)}+\|\nabla\varphi_k\|_{L^2(\omti)}\right)^\tau
    \text{ for }k=0,1,
\end{equation}
and the trace estimate, for $1/2<\tau<1$,
\begin{equation}\label{s4-7}
     \|\varphi_k\|_{L^2(\der T_j^{(0)})}\leq \|\varphi_k\|_{H^{\tau-1/2}(\der T_j^{(0)})}\leq C_\tau \|\varphi_k\|_{H^\tau(\omti)},\text{ for }k=0,1.
\end{equation}
The estimates \eqref{s3-7} and \eqref{s4-7} combined with \eqref{e3.1bis} give, for $k=0,1$,
\begin{equation}\label{s1-8}
  \|\varphi_k\|_{L^2(\der T_j^{(0)} )}\leq C_\tau \mu^{\tau-1},
\end{equation}
and, hence,
\begin{equation}\label{ssnum-8}
\int_{\der T_j^{(0)}}\left(|\varphi_0|+|\varphi_1|+|\varphi_0||\varphi_1|\right)\leq C_\tau \mu^{2(\tau-1)}
\end{equation}
for any fixed $\tau \in (1/2,1)$. By using \eqref{s1-7},
\eqref{e4.05} and \eqref{ssnum-8} we have the estimate
\begin{equation}\label{s4-9}
   \left|\sum_{j=1}^N q_j
   \int_{\der T_j^{(0)}} e^{i x \cdot \xi}
         (\tilde{\Phi}_j(x) \nu_j)\right|
   \leq C \left(m_0 e^{C(|\xi|+\mu)} + \mu^{-2 (1-\tau)}\right).
\end{equation}
We write the integral on the left-hand side of \eqref{s4-9} in a
slightly different form. We denote by $\left \{ F_{k}\right \}
_{k=1}^{M_1}$ the collection of facets of tetrahedra. We note that the
set $\bigcup_{k=1}^{M_1}F_{k}$ contains special a priori information
which is implied by the a priori information on the mesh of
tetrahedra.

Each facet $F_{k}$ not contained on $\der\om$ belongs to two
tetrahedra and the outer normal directions with respect to these two
tetrahedra are opposite one to another. We denote by $\nu_k$ one of
these two directions and denote by $q_k^-$ the coefficient defined in
the tetrahedron where $\nu_k$ is pointing towards and $q_k^+$ the one
defined in the other tetrahedron. By assumption \eqref{3.4} and by
\eqref{1.3} we have that
\begin{equation}\label{contr}
   |q_k^+-q_k^-| \geq c_0.
\end{equation}
For any $k\in \{1,\ldots ,M_1\}$ we let
\begin{equation}\label{fk}
    f_{k}(x) = \begin{cases}
    0 & \text{ if }F_k\text{ is contained in }\der\om
\\
    \left(q_{k}^{+} - q_{k}^{-}\right)
           (\tilde{\Phi}_{k}(x) \cdot \nu_k) &
        \text{ otherwise.}\end{cases}
\end{equation}
We know that the $f_{k}$ are affine functions on each facet, $F_{k}$,
and that
\begin{equation}\label{ap1}
   \sum_{k=1}^{M_1} \left \Vert f_{k}\right \Vert_{
          H^{1/2}\left( F_{k}\right) }\leq E,
\end{equation}
where $E$ depends on a priori information. We denote by $H$ the
measure,
\[
   H = \sum_{k=1}^{M_1}h_k := \sum_{k=1}^{M_1} f_{k} d\sigma_{k},
\]
where $d\sigma _{k}$ is the surface element on $F_{k}$ for $k\in
\{1,\ldots ,M_1\}$. More precisely, each $h_k$ is defined as follows:
\[
   C_{0}^{0}\left(\mathbb{R}^{3}\right) \ni \phi
           \rightarrow \left \langle h_k,\phi \right \rangle
       = \int f_{k} \phi \, d\sigma_{k} \in \mathbb{R}.
\]

Estimate \eqref{s4-9} implies that
\begin{equation}
\label{Hh}
   \vert \widehat{H}(\xi ) \vert
                  \leq C \gamma\left(|\xi|,\mu,m_{0}\right),
\end{equation}
where
\begin{equation}\label{Hh1}
   \gamma\left(t,\mu,m_0\right)
       = m_{0}e^{C(t+\mu)}+\mu^{-2(1-\tau )}
                          \text{ for every }t>0,\mu>0.
\end{equation}
We estimate, for $s>1$,
\begin{equation}\label{h-s}
   \left( \int_{\mathbb{R}^{3}}\left(1+|\xi|^{2}\right)^{-s/2}
         \vert\widehat{H}(\xi )\vert^{2} d\xi \right)^{1/2}
   \leq \sum _{k=1}^{M_1} \left(\int_{\mathbb{R}^{3}}
   \left(1+|\xi|^{2}\right)^{-s/2}
          \vert\widehat{h}_{k}(\xi)\vert^{2}d\xi \right)^{1/2}.
\end{equation}
For each $k$ we write
\begin{eqnarray}\label{h-s2}
  \int_{\mathbb{R}^{3}} \left(1+|\xi|^{2}\right)^{-s/2}
        \vert\widehat{h}_{k}(\xi)\vert^{2} d\xi
  &=& \int_{|\xi|\leq 1} \left(1+|\xi|^{2}\right)^{-s/2}
        \vert\widehat{h}_{k}(\xi)\vert^{2} d\xi
\nonumber\\
&& \hspace*{-2.0cm}
   + \sum_{j=1}^\infty \int_{2^j\leq|\xi|\leq 2^{j+1}}
     \left(1+|\xi|^{2}\right)^{-s/2}
            \vert\widehat{h}_{k}(\xi)\vert^{2} d\xi
\nonumber\\[-0.1cm]
   &\leq& \int_{|\xi|\leq1} \vert\widehat{h}_{k}(\xi)\vert^{2}d\xi
     + \sum_{j=1}^\infty 2^{-js} \int_{|\xi| \leq 2^{j+1}}
                     \vert\widehat{h}_{k}(\xi)\vert^{2} d\xi.
\end{eqnarray}
Using \cite[Theorem 7.1.26, p.173]{Ho}, estimate \eqref{h-s2} gives
\begin{equation*}
   \int_{\mathbb{R}^{3}}\left(
1+|\xi |^{2}\right)^{-s/2} \vert\widehat{h}_{k}(\xi)\vert ^{2}d\xi \leq C\left(1+2\sum_{j=1}^\infty2^{-(s-1)j}\right)\int_{F_k}\vert f_k\vert^2d\sigma_k,
\end{equation*}
 and, by \eqref{ap1} and \eqref{h-s},
\begin{equation}\label{10}
   \left(\int_{\mathbb{R}^{3}} \left(1+|\xi|^{2}\right)^{-s/2}
   \vert\widehat{H}\left(\xi\right)\vert^{2} d\xi\right)^{1/2}
             \leq C E.
\end{equation}

We consider a single facet, for instance, the facet $F_{1}$. To
simplify the notation, we assume that $F_{1}\subset
\mathbb{R}^{2}\times \{0\}$ and that $0$ is a point of $F_{1}$ such
that $B_{2d}^{\prime }\left( 0\right) \subset F_{1}$ where $d$ depends
on the a priori information only.  We let $\eta \in C_{0}^{\infty
}\left( \mathbb{R}^{2}\right)$ such that $0 \leq \eta \leq 1$ and
$\eta =1$ on $B_{d}^{\prime }\left( 0\right)$.

We choose a $g_{1}\in H^{1}\left( \mathbb{R}^{3}\right) $ such that
\[
g_{1}(x^{\prime },0)=\left( \eta f_{1}\right) (x^{\prime })\text{%
, }x^{\prime }\in \mathbb{R}^{2},
\]%
\begin{equation}
\label{support}
\text{supp}g_{1}\cap F_{k}=\emptyset \text{, for }k\neq 1
\end{equation}
and
\begin{equation}
\label{extens1}
\left \Vert g_{1}\right \Vert _{H^{1}\left( \mathbb{R}^{3}\right)
}\leq CE,
\end{equation}
where $C$ depends on the a priori information only. Taking into
account (\ref{support}), we obtain
\begin{eqnarray*}
\int _{\mathbb{R}^{3}}\widehat{H}(\xi )\check{g}_1 (\xi )d\xi
&=&\int_{\mathbb{R}^{3}}d\xi \check{g}_1 (\xi
)\sum_{k=1}^{M_1}\int_{F_{k}}e^{ix\cdot \xi }f_{k}(x)d\sigma
_{k}
\\&=&\sum _{k=1}^{M_1}\int_{F_{k}}f_{k}(x)d\sigma
_{k}\int_{\mathbb{R}^{3}}e^{ix\cdot \xi }\check{g}_1 (\xi )d\xi
\\
&=&\sum_{k=1}^{M_1}\int_{F_{k}}f_{k}(x)
g_{1}(x)d\sigma _{k}=\int_{F_{1}}\eta \left \vert
f_{1}\right \vert ^{2}d\sigma _{1},
\end{eqnarray*}
that is,
\begin{equation}\label{20}
\int_{F_{1}}\eta \left \vert f_{1}\right \vert ^{2}d\sigma
_{1}=\int_{\mathbb{R}^{3}}\widehat{H}(\xi )\check{g}_1 (\xi )d\xi.
\end{equation}
Moreover by (\ref{extens1}) we have
\begin{equation}
\label{15}
\int_{\mathbb{R}^{3}}\left( 1+|\xi |^{2}\right) \left \vert \check{g}_1
(\xi )\right \vert ^{2}d\xi \leq CE^{2}.
\end{equation}%

We write
\begin{equation}
\label{27}
\int_{\mathbb{R}^{3}}\left \vert \widehat{H}(\xi )\check{g}_1 (\xi
)\right \vert d\xi=\int_{|\xi |\leq \rho }\left \vert \widehat{H}
(\xi )\check{g}_1 (\xi )\right \vert d\xi +\int_{|\xi |>\rho }\left \vert
\widehat{H}(\xi )\check{g}_1 (\xi )\right \vert d\xi 
\end{equation}
By (\ref{Hh}) and (\ref{15}) we have
\begin{eqnarray}\label{27.5}
&&\int_{|\xi |\leq \rho }\left \vert \widehat{H}
(\xi )\check{g}_1 (\xi )\right \vert d\xi\leq\gamma \left( \rho ,\mu,m_{0}\right)
\int_{|\xi |\leq \rho } \left \vert
\check{g}_1 (\xi )\right \vert d\xi\nonumber\\
&& \quad \leq C\gamma \left( \rho ,\mu,m_{0}\right) \left(\int_{|\xi |\leq \rho }\!\!\!
(1+|\xi|^2)^{-1}d\xi\right)^{1/2}\left(\int_{|\xi |\leq \rho }\!\!\!(1+|\xi|^2)\left\vert\check{g}_1 (\xi )\right \vert^2 d\xi\right)^{1/2}
\nonumber\\
&& \quad \leq C\gamma \left( \rho ,\mu,m_{0}\right)\sqrt{\rho}E
\end{eqnarray}
Using the Cauchy-Schwarz inequality and (\ref{10}) we have
\begin{eqnarray}
\label{30}\nonumber
&\int_{|\xi |>\rho }\left \vert \widehat{H}(\xi )\check{g}_1 (\xi
)\right \vert d\xi=
\int_{|\xi |>\rho }\left( 1+|\xi |^{2}\right)
^{-s/4}\left \vert \widehat{H}(\xi )\right \vert \left( 1+|\xi |^{2}\right)
^{s/4}\left \vert \check{g}_1 (\xi )\right \vert d\xi &
\\ \nonumber
&\leq \left( \int_{|\xi |>\rho }\left( 1+|\xi |^{2}\right)
^{-s/2}\left \vert \widehat{H}(\xi )\right \vert ^{2}d\xi \right) ^{1/2}\left(
\int_{|\xi |>\rho }\left( 1+|\xi |^{2}\right) ^{s/2}\left \vert
\check{g}_1 (\xi )\right \vert ^{2}d\xi \right) ^{1/2}&
\\
&\leq CE\left( \int_{|\xi |>\rho }\left( 1+|\xi |^{2}\right)
^{s/2}\left \vert \check{g}_1 (\xi )\right \vert ^{2}d\xi \right) ^{1/2}.&
\end{eqnarray}
Then, using (\ref{15}), we find that for $1<s<2$,
\begin{eqnarray}
\label{40}
&&\int_{|\xi |>\rho }\left( 1+|\xi |^{2}\right) ^{s/2}\left \vert
\check{g}_1 (\xi )\right \vert ^{2}d\xi
\\ \nonumber
&& \quad
   = \int_{|\xi |>\rho }\left(
1+|\xi |^{2}\right) ^{-\frac{2-s}{2}}\left( 1+|\xi |^{2}\right) \left \vert
\check{g}_1 (\xi )\right \vert ^{2}d\xi
\\ \nonumber
&& \quad
   \leq \left( 1+\rho ^{2}\right) ^{-\frac{2-s}{2}}\int_{\mathbb{R}
^{3}}\left( 1+|\xi |^{2}\right) \left \vert \check{g}_1 (\xi )\right \vert
^{2}d\xi \leq CE^{2}\rho ^{-(2-s)}
\end{eqnarray}
and with (\ref{30}) and (\ref{40}),
\begin{equation}
\int \limits_{|\xi |>\rho }\left \vert \widehat{H}(\xi )\check{g}_1 (\xi
)\right \vert d\xi \leq CE^2\rho ^{-\frac{2-s}{2}}.  \label{50}
\end{equation}
By (\ref{20}), (\ref{27}), \eqref{27.5} and (\ref{50}) we have
\begin{equation}\label{60}
\int_{F_{1}}\eta \left \vert f_{1}\right \vert ^{2}d\sigma _{1}\leq
CE\sqrt{\rho} \left( m_{0}e^{C\rho }e^{C\mu}+\mu^{-2(1-\tau )}\right) +CE^2\rho
^{-\frac{2-s}{2}}.
\end{equation}
We choose $\mu=\rho ^{1/(1-\tau )}$ and get, for every $\rho \geq 1$,
\begin{eqnarray}
\label{70}
&&\int_{F_{1}}\eta \left \vert f_{1}\right \vert ^{2}d\sigma _{1}\leq
CE\left( m_{0}\sqrt{\rho} e^{C\left( \rho +\rho ^{1/(1-\tau )}\right) }+\rho
^{-3/2}+E\rho ^{-\frac{2-s}{2}}\right)
\\ \nonumber
&& \quad
   \leq C\left( E+m_{0}+1\right)^2 \left(
\left( \frac{m_{0}}{E+m_{0}+1}\right) e^{C_{\star}\rho ^{1/(1-\tau )}}+\rho ^{-
\frac{2-s}{2}}\right),
\end{eqnarray}
where $C_\star$ depends on the a priori data only.

We then choose
\[
   \rho = \left(\frac{1}{2C_{\star}}
        \left\vert \log \frac{m_{0}}{E+m_{0}+1} \right\vert
                     \right)^{1-\tau }
\]
so that
\begin{equation}\label{81}
   \int_{B_{d}^{\prime}} \vert f_{1} \vert ^{2} d\sigma_{1}
   \leq C  \left(E + m_{0} + 1\right)^2
         \left\vert \log \frac{m_{0}}{E + m_{0} + 1}
                            \right\vert^{-\frac{2-s}{2}},
\end{equation}
where $C$ depends on $s$ and the a priori information only. Because
$f_1$ is an affine function on $F_1$ with a bounded gradient, and the
size of $B_{d}^{\prime }$ is bounded from below with a constant
depending only on a priori information, we have
\begin{equation}\label{82}
\left \vert f_{1}(x)\right \vert\leq C\left( E+m_{0}+1\right) \left \vert \log \frac{m_{0}}{E+m_{0}+1}
\right \vert ^{-\frac{2-s}{4}}\text{ for every }x\in F_1.
\end{equation}
By repeating the same procedure on each facet, and recalling
\eqref{contr} and the fact that $\tilde{\Phi}_k(x)\cdot \nu_k=0$ if
$F_k \subset \der\om$, we have
\begin{equation}\label{G1}
    \left\vert\tilde{\Phi}_k(x)\cdot \nu_k\right\vert\leq C\varsigma_1(m_0) \text{ for }x\in F_k,
\end{equation}
where
\[
\varsigma_1(m_0)=\left( E+m_{0}+1\right)\left \vert \log \frac{m_{0}}{E+m_{0}+1}
\right \vert ^{-\frac{2-s}{4}}.
\]
We fix a tetrahedron $T_j^{(0)}$ and let $P_1$, $P_2$, $P_3$, $P_4$ be
its vertices. We label the facets so that $F_k$, for $k=1,2,3,4$ is the
facet of $T_j^{(0)}$ that does not contain $P_k$. We let $\nu^{(k)}$ be
the unit outward normal to $F_k$.  Each point on $x\in F_k$ can be
written as
\[
   x = \sum_{i=1}^4 s_i P_i,
\]
where $0\leq s_i\leq 1$, $s_k=0$, and $\sum_{i=1}^4s_i=1$.  With this
notation,
\[
   \tilde{\Phi}_k(x) \cdot \nu_k
           = \sum_{i=1}^4 s_i \tilde{v}_i \cdot \nu^{(k)}
\]
and using \eqref{G1}
\[
   \left\vert \sum_{i=1}^4s_i\tilde{v}_i\cdot \nu^{(k)} \right\vert\leq C\varsigma_1(m_0).
\]
This implies that
\[
   \left\vert \tilde{v}_i\cdot \nu^{(k)} \right\vert
         \leq  C\varsigma_1(m_0)\text{ for every }i\neq k.
\]

In particular, this means that for every vector $\tilde{v}_{j,i}$ we
have
\[
   \left\vert \tilde{v}_{j,i} \cdot \nu^{(k)} \right\vert
            \leq C \varsigma_1(m_0)
\]
for every direction $\nu_{j}^{(k)}$ orthogonal to the facet of
$T_j^{(0)}$ that contains $P_i$. By the regularity of the partition,
this implies that
\begin{equation}\label{G2}
   \left\vert \tilde{v}_{j,i} \right\vert
                  \leq C_3 \varsigma_1(m_0),
\end{equation}
where $C_3$ depends on the a priori information.

By adding together inequalities \eqref{G2} and applying
\eqref{normalizz}, we get
\[1=\sum_{j=1}^N\sum_{i=1}^4\left\vert\tilde{v}_{j,i}\right\vert\leq 4C_3\varsigma_1(m_0)\]
that yields
\begin{equation}\label{G3}
    m_0\geq \varsigma_1^{-1}\left(\frac{1}{4C_3}\right).
\end{equation}
From the definition of $m_0$ (see \eqref{m0}), there exist a pair of
boundary values $\phi_0$ and $\psi_0$ such that
\[\left\vert\mathcal{G}(\phi_0,\psi_0)\right\vert\geq \frac{m_0}{2}\|\phi_0\|\|\psi_0\|\]
and, hence,
\[\left|\frac{d}{dt}\mathcal{F}(t,\phi_0,\psi_0)_{|_{t=0}}\right\vert\geq \omega^2\sum_{j=1}^NV_j\frac{m_0}{2}\|\phi_0\|\|\psi_0\|\]
that, together with \eqref{s1}, gives \eqref{thirdstep} for
\[m_1=\frac{1}{2}\omun A_2^{-1}\varsigma_1^{-1}\left(\frac{1}{4C_3}\right).
\]

\bibliography{deformation}

\begin{thebibliography}{10}

\bibitem{A-personal-communication}
G.~Alessandrini.
\newblock Personal communication.

\bibitem{A}
G.~Alessandrini.
\newblock Stable determination of conductivity by boundary measurements.
\newblock {\em Appl. Anal.}, 27:153--172, 1988.

\bibitem{AV}
G.~Alessandrini and S.~Vessella.
\newblock Lipschitz stability for the inverse conductivity problem.
\newblock {\em Adv. in Appl. Math.}, 35:207--241, 2005.

\bibitem{BdHQ}
E.~Beretta, M.~de~Hoop, and L.~Qiu.
\newblock Lipschitz stability of an inverse boundary value problem for a
  {S}chr{\"{o}}dinger type equation.
\newblock {\em SIAM J. Math. Anal.}, 45(2):679--699, 2013.

\bibitem{BdHQS}
E.~Beretta, M.~de~Hoop, L.~Qiu, and O.~Scherzer.
\newblock Inverse boundary value problem for the {H}elmholtz equation:
  {M}ultilevel approach and interative reconstruction.
\newblock {\em ArXiv.org/pdf/1406.2391.pdf}, 2014.

\bibitem{BF}
E.~Beretta and E.~Francini.
\newblock Lipschitz stability for the impedance tomography problem. {T}he
  complex case.
\newblock {\em Comm. PDE}, 36:1723--1749, 2011.

\bibitem{dHQS1}
M.~de~Hoop, L.~Qiu, and O.~Scherzer.
\newblock Local analysis of inverse problems: {H}{\"{o}}lder stability and
  iterative reconstruction.
\newblock {\em Inverse Problems}, 28, 2012.

\bibitem{dHQS2}
M.~de~Hoop, L.~Qiu, and O.~Scherzer.
\newblock A convergence analysis of a multi-level projected steepest descent
  iteration for nonlinear inverse problems in banach spaces subject to
  stability constraints.
\newblock {\em Numerische Mathematik}, in print, 2014.

\bibitem{H}
P.~H{\"{a}}hner.
\newblock A periodic faddeev-type solution operator.
\newblock {\em J. Differential Equations}, 128(1):300--308, 1996.

\bibitem{Ho}
{L. H\"{o}rmander}.
\newblock {\em The Analysis of Linear Partial Differential Operators I}.
\newblock Springer-Verlag, 1983.

\bibitem{MP}
Magnanini and Papi.
\newblock An inverse problem for the {H}elmholtz equation.
\newblock {\em Inverse Problems}, 1(4):357--370, 1985.

\bibitem{M}
N.~Mandache.
\newblock Exponential instability in an inverse problem for the
  {S}chr{\"{o}}dinger equation.
\newblock {\em Inverse Problems}, 17:1435--1444, 2001.

\bibitem{Nov}
R.~Novikov.
\newblock New global stability estimates for the {G}el'fand-{C}alder\'{o}n
  inverse problem.
\newblock {\em Inverse Problems}, 27(1), 2011.

\bibitem{Pratt1999}
R.~G. Pratt.
\newblock Seismic waveform inversion in the frequency domain, part 1: Theory
  and verification in a physical scale model.
\newblock {\em Geophysics}, 64(3):888--901, 1999.

\bibitem{Pratt1998}
R.~G. Pratt, C.~Shin, and G.~Hicks.
\newblock Gauss-{N}ewton and full {N}ewton methods in frequency-space seismic
  waveform inversion.
\newblock {\em Geophysical Journal International}, 133(2):341--362, 1998.

\bibitem{RugerHaleGeophysics:2006}
A.~R{\"{u}}ger and D.~Hale.
\newblock Meshing for velocity modeling and ray tracing in comples velocity
  fields.
\newblock {\em Geophysics}, 71(1):U1--U11, 2006.

\bibitem{S}
M.~Salo.
\newblock Lecture notes on the {C}alder\`{o}n problem.
\newblock {\em unpublished ed.}

\bibitem{Hilst}
F.~Simons, A.~Zielhuis, and R.~Van~der Hilst.
\newblock The deep structure of the {A}ustralian continent inferred from
  surface wave tomography.
\newblock {\em Lithos}, 48(1-4):17--43, 1999.

\bibitem{Pratt2004}
L.~Sirgue and R.~G. Pratt.
\newblock Efficient waveform inversion and imaging : A strategy for selecting
  temporal frequencies.
\newblock {\em Geophysics}, 69(1):231--248, 2004.

\bibitem{SU}
J.~Sylvester and G.~Uhlmann.
\newblock A global uniqueness theorem for an inverse boundary value problem.
\newblock {\em Ann. of Math. (2)}, 125(1):153--169, 1987.

\bibitem{Virieux2009}
J.~Virieux and S.~Operto.
\newblock An overview of full-waveform inversion in exploration geophysics.
\newblock {\em Geophysics}, 74:WCC1--WCC26, 2009.

\end{thebibliography}
\bibliographystyle{abbrv}

\end{document}